\documentclass[11pt,reqno]{amsart}

\usepackage{setspace}
\usepackage{footnote}
\usepackage{graphicx}
\usepackage{cmap,mathtools}
\usepackage[T1]{fontenc}
\usepackage[utf8]{inputenc}
\usepackage[english]{babel}
\usepackage{amsfonts,amssymb,amsthm,amsmath,enumerate,bm,cite, mathrsfs}
\usepackage{url}
\usepackage[shortlabels]{enumitem}
\usepackage{float}
\usepackage{secdot}

\allowdisplaybreaks
\usepackage{graphicx}
\usepackage{epstopdf}
\usepackage{a4wide}
\usepackage{xparse}
\usepackage[dvipsnames]{xcolor}
\usepackage[colorlinks=true,linktocpage,pdfpagelabels,bookmarksnumbered,bookmarksopen]{hyperref}
\hypersetup{
	colorlinks=true,
	linkcolor=blue,         
	citecolor=red,
	urlcolor=black
}

\usepackage[margin=1in]{geometry}

\usepackage{orcidlink}  
\mathtoolsset{showonlyrefs} 

\newtheorem{theorem}{Theorem}

\newtheorem{lemma}[theorem]{Lemma}
\newtheorem{corollary}[theorem]{Corollary}
\theoremstyle{definition}

\newtheorem{remark}[theorem]{Remark}
\newtheorem{conjecture}[theorem]{Conjecture}

\let\OLDthebibliography\thebibliography
\renewcommand\thebibliography[1]{
	\OLDthebibliography{#1}
	\setlength{\parskip}{1pt}
	\setlength{\itemsep}{1pt plus 0.3ex}
}

\numberwithin{equation}{section} 
\numberwithin{theorem}{section}  

\DeclarePairedDelimiter\norm{\lVert}{\rVert}%
\makeatletter
\let\oldnorm\norm
\def\norm{\@ifstar{\oldnorm}{\oldnorm*}}

\DeclareMathOperator*{\curl}{curl}

\newcommand{\pa} {\partial}

\newcommand{\om} {\omega}
\newcommand{\Om} {\Omega}

\newcommand{\la} {\lambda}

\newcommand\restr[2]{{
  \left.\kern-\nulldelimiterspace 
  #1 
  \right|_{#2} 
  }}
  
\def\Ab{{\mathbf{A}}}

\def\Fb{{\mathbf{F}}}

\def\i{{\mathrm{i}}}
\def\FAB{\mathbf{F}^{\rm AB}}

\def\R{{\mathbb R}}
\def\Z{{\mathbb Z}}

\def\({{\Big(}}
\def\){{\Big)}}

\def\d{{\rm d}}

\def\dx{{\rm d}x}

\usepackage[foot]{amsaddr}
\usepackage[pagewise]{lineno}

\makeatother

\date{}
\begin{document}
\title[Isoperimetric Inequality with zero magnetic field]{Isoperimetric Inequality with zero magnetic field in doubly connected domains}

\author[M. Ghosh]{Mrityunjoy Ghosh$^\dagger$\orcidlink{0000-0003-0415-2821}}

\address{$^\dagger$Tata Institute of Fundamental Research,
Centre for Applicable
Mathematics\\
Post Bag No. 6503, Sharadanagar,
Bangalore 560065, India.}
\email{ghoshmrityunjoy22@gmail.com}

\author[A. Kachmar]{Ayman Kachmar$^\ddagger$}
\address{$^\ddagger$School of Science and Engineering, The Chinese University of Hong Kong, Shenzhen, Guangdong, 518172,
P.R. China.}
\email{akachmar@cuhk.edu.cn}



\begin{abstract}
We investigate how the lowest eigenvalue of a magnetic Laplacian depends on the geometry of a planar domain with a disk shaped hole, where the magnetic field is generated by a singular flux. Under Dirichlet boundary conditions on the inner boundary and Neumann boundary conditions on the outer boundary, we show that this eigenvalue is maximized when the domain is an annulus, for a fixed area and magnetic flux. As consequences, we establish geometric inequalities for eigenvalues in settings with both singular and localized magnetic fields. We also propose a conjecture for a general optimality result and establish its validity for large magnetic fluxes.
\end{abstract}

\maketitle


\section{Introduction}

The spectral optimization of the lowest eigenvalue of differential operators is a fundamental topic in spectral geometry, with the Faber-Krahn inequality for the Dirichlet Laplacian being a paradigmatic result (cf. \cite{Faber, Krahn1925, Krahn1926}). In the magnetic setting, a landmark contribution is due to Erd\"os \cite{Erdos1996}, who established a magnetic analogue of the Faber-Krahn inequality in the plane. Specifically, for the magnetic Laplacian with Dirichlet boundary conditions under a homogeneous magnetic field, the disk was shown to minimize the first eigenvalue among all planar domains of prescribed area. Recently, Ghanta-Junge-Morin \cite{GJM2024} proved a quantitative version of Erd\"os's result.

Fournais and Helffer \cite{FournaisHelffer2019} conjectured a reverse Faber-Krahn inequality for the Neumann Laplacian with a homogeneous magnetic field, specifically for simply connected planar domains of prescribed area. For moderate magnetic field strengths, this conjecture was recently resolved in \cite{CLPS} (see \cite{KachmarLotoreichik} for a quantitative version involving convex domains).

Zero magnetic fields in multiply connected domains are known to generate spectral flux effects. However,  their role in shape optimization remains less explored. In this paper, we prove a reverse Faber-Krahn inequality for doubly connected domains with mixed Neumann-Dirichlet conditions and zero field, partially extending the work of Colbois-Provenzano-Savo \cite{ColbProvenSavo} on point-perforated domains.

\subsection*{Mathematical setup} Consider two bounded and simply connected domains $\om,\Om\subset \R^2$ with $C^1$ smooth boundaries and such that $\overline{\om}\subset \Om$.
Let $\Om_0=\Om\setminus \overline{\om}$ and, for a given $\Phi\in\mathbb R$, consider a vector potential $\Fb_\Phi$ satisfying
\[\curl\Fb_\Phi=0\mbox{ on }\Om_0,\quad \frac1{2\pi}\int_{\partial\omega}\Fb_\Phi\cdot \dx=\Phi. \]
Applying Stoke's theorem formally,
\[ \int_{\partial\om}\Fb_\Phi\cdot \dx=\int_{\om}\curl\Fb_\Phi\,\dx,\]
we can interpret  the circulation $\Phi$ of $\Fb_\Phi$ along $\partial\om$ as the (magnetic) flux across $\om$, and even across $\Om$ since $\curl\Fb_\Phi$ vanishes outside $\om$.

We introduce the eigenvalue
\begin{equation}\label{LaOmo}
    \la^{\Om_0}(\Phi)=\inf_{\substack{u\in H^1(\Om_0)\\u|_{\pa \om}=0}} \frac{\|(-\i\nabla -\Fb_\Phi)u\|^2_{L^2(\Om_0)}}{\|u\|_{L^2(\Om_0)}^2},
\end{equation}
which depends only on the circulation of $\Fb_\Phi$ along $\partial\om$ and not on the specific choice of $\Fb_\Phi$ (see \cite[Theorem~1.1]{HHOO}). Moreover, it is an even  periodic function of $\Phi$ with period $1$. 
That  said, it is no loss of generality to choose $\Fb_\Phi$ as
\begin{equation}\label{eq:def-F}
\Fb_\Phi(x)=\FAB_\Phi(x-x_0)
\end{equation}
where $x_0$ is any  point in $\om$ and $\FAB_\Phi$ is the Aharonov-Bohm vector potential with flux $\Phi$ defined as 
\[\FAB_\Phi(x)=\Phi\FAB_1(x),\quad \FAB_1(x)=\left(\frac{-x_2}{|x|^2},\frac{x_1}{|x|^2}\right),\;x=(x_1,x_2).\]
\subsection*{Main result}
We would like to study the dependence of the eigenvalue $\lambda^{\Om_0}(\Phi)$ on the shape of $\Om_0$ by fixing the area $|\Om_0|$ and the flux $\Phi$ (this has been studied in \cite{ColbProvenSavo} in the case where $\om=\{x_0\}$). Specifically, we introduce the optimal domains
\begin{equation}\label{eq:def-opt-dom}
\Om_0^*=\Om^*\setminus\overline{\om^*},\quad \Om^*=\{x\in\R^2\colon |x|<R_1\},\quad \om^*=\{x\in\R^2\colon |x|<R_0\},
\end{equation}
with the radii $R_0$ and $R_1$ fixed according to
\begin{equation}\label{eq:constraint}
|\Om|=|\Om^*|,\quad |\om|=|\om^*|. 
\end{equation}
For technical reasons, we restrict to a disk shaped hole  and assume
\begin{equation}
    \label{eq:cond-geom}
    \om=\{x\in\R^2\colon |x-x_0|<R_0\}.
\end{equation}
\begin{theorem}\label{theo:FB}
    Suppose that \eqref{eq:constraint} and \eqref{eq:cond-geom} hold. If $\Om_0$ is not an annulus, then we have 
    \begin{equation}
        \la^{\Om_0}(\Phi)< \la^{\Om_0^*}(\Phi),
    \end{equation}
    where $\Om_0^*$ is introduced in \eqref{eq:def-opt-dom}.
\end{theorem}

\begin{remark}\label{rem:Phi=0}
If $\Phi$ is an integer, then $\lambda^{\Om_0}(\Phi) = \lambda^{\Om_0}(0)$ due to periodicity, reducing the problem to the Laplacian without a magnetic field. The isoperimetric inequality for the Laplacian is well studied under various geometric constraints. For doubly connected domains, we refer to \cite{Hersch1963} and \cite{Payne-Weinberger1961}, where the area of the domain and the perimeter of the Dirichlet boundary component are fixed. We also refer to the recent works \cite{AAK,AM,DP,ABD2023} for various  generalizations of these results in the context of the Laplacian/$p$-Laplacian with mixed boundary conditions. In particular, recently in \cite{ABD2023}, the authors investigated a similar problem for the Laplacian with Robin-Neumann boundary conditions under the assumption that the inner hole is disk shaped.
\end{remark}
\subsection*{Applications}
Thanks to a limiting argument similar to \cite[Eq. (5.11)]{KP}, if  
\[\om=\om_\epsilon:=\{x\in\R^2\colon |x-x_0|<\epsilon\}\] and we take $\epsilon\to0$, we have $\lambda^{\Om_0}(\Phi)\to \lambda^{\Om}(\Phi)$, where 
\[\lambda^{\Om}(\Phi):=\inf_{u\in C_c^\infty(\Omega\setminus\{x_0\})}\frac{\|(-\i\nabla -\Fb_\Phi)u\|^2_{L^2(\Om)}}{\|u\|_{L^2(\Om)}^2}\] 
is the lowest eigenvalue of the Friedrichs extension of the magnetic Laplacian $(-\i\nabla-\Fb_\Phi)^2$ in $L^2(\Om)$. We then obtain as an immediate corollary of Theorem~\ref{theo:FB} the isoperimetric inequality with Aharonov-Bohm vector  potential\footnote{The proof of Theorem~\ref{theo:FB} uses a trial state constructed in the spirit of \cite{ColbProvenSavo}.} proved in \cite{ColbProvenSavo}.
\begin{corollary}\label{corol:FB} It holds
    \[\lambda^{\Om}(\Phi)\leq \lambda^{\Om^*}(\Phi).\]
\end{corollary}
Another corollary of Theorem~\ref{theo:FB} is the isoperimetric inequality with a localized magnetic field. To state this inequality, let us choose a vector potential $\Ab=\Ab_\Phi^{\Om,\om}$ on $\Om$ satisfying
\begin{equation}\label{eq:def-A}
     \frac1{2\pi}\int_\Omega\curl\Ab\,\dx=\Phi\quad\mbox{and}\quad \curl\Ab=b\mathbf 1_{\omega},
\end{equation}
with $\mathbf 1_\om$ the characteristic function of $\omega$ and $b$ a positive constant. The condition on the flux leads to the choice $b=2\pi\Phi/|\omega|$, and it turns out  that fixing $\Phi$ is equivalent to fixing the area of $\omega$ and the field's intensity $b$.

With the vector potential $\Ab_\Phi^{\Om,\om}$ as above, we introduce the eigenvalue
\begin{equation}\label{eq:def-la-loc-f}
    \la (\Phi, \Om, \om)=\inf_{u\in H^1(\Om)\setminus \{0\}} \frac{\|(-\i\nabla -\Ab_\Phi^{\Om,\om})u\|^2_{L^2(\Om)}}{\|u\|_{L^2(\Om)}^2},
\end{equation}
which depends only on $\curl\Ab_\Phi^{\Om,\om}$ by gauge invariance, since the domain $\Om$ is simply connected. Recall that, as $\Phi\to+\infty$, we have by \cite[Theorem~1.4]{KachmarSundqvist2024}
\begin{equation}\label{eq:lim-la-loc-f}
\la (\Phi, \Om, \om)\sim \la^{\Om_0}(\Phi).
\end{equation}
We then have the following corollary of Theorem~\ref{theo:FB}.
\begin{corollary}\label{corol:loc-field}
    Let $\nu\in[0,1)$ and suppose that  \eqref{eq:constraint} and \eqref{eq:cond-geom} hold. There exists $\Phi_0>0$ such that, if $\Phi\geq \Phi_0$ and the fractional part\footnote{This is $\{\Phi\}=\Phi-\lfloor\Phi\rfloor$.} of $\Phi$ is equal to $\nu$, then
    \[ \la(\Phi,\Om,\om)\leq \la(\Phi,\Om^*,\om^*),\]
    where $\Om^*$ and $\om^*$ are as in \eqref{eq:def-opt-dom}.
\end{corollary}
The constant $\Phi_0$ in Corollary~\ref{corol:loc-field} depends on the fixed domains $\Om,\om$. In that regard, the inequality stated is analogous to the inequalities in \cite{FournaisHelffer2019}. This motivates us to state the following conjecture. 
\begin{conjecture}\label{Conjecture}
    Let $\Phi\in\R$. Suppose that $\Om$ and $\om$ satisfy \eqref{eq:constraint}. Then the following inequality holds
    \begin{equation}
        \la (\Phi, \Om, \om)\leq \la (\Phi, \Om^*, \om^*).
    \end{equation}
    \end{conjecture}
We also suspect that the equality holds if and only if $(\Om,\om)=(\Om^*,\om^*)$ modulo a translation.

\subsection*{Organization}
The rest of the article is organized as follows. In the next section (Section~\ref{sec:Main}), we discuss the eigenvalues and the eigenfunctions of the magnetic Laplacian on the annulus with mixed boundary conditions: We prove the existence of a radial ground state and establish 
some monotonicity results associated with it.
Section~\ref{sec:proof_thm} is devoted to the proof of Theorem \ref{theo:FB}. 
Finally, in Section~\ref{sec:Neumann}, we extend Theorem~\ref{theo:FB} to the case of (magnetic) Neumann boundary conditions (see Theorem \ref{theo:Neu}) on both the interior and exterior boundaries.

\section{The case of the annulus}\label{sec:Main}

In this section, we study the case of the annulus $\Omega_0^*$ with inner radius $R_0$ and outer radius 
$R_1$. The eigenvalue $\lambda^{\Omega_0}(\Phi)$ is $1$-periodic in $\Phi$, so we restrict $\Phi$ to the interval $(-1/2, 1/2]$. Furthermore, the unitary transformation $u \mapsto \overline{u}$ implies $\lambda^{\Omega_0^*}(\Phi) = \lambda^{\Omega_0^*}(-\Phi)$, allowing us to further restrict $\Phi$ to $[0, 1/2]$. Under this restriction, we will show that $\lambda^{\Omega_0^*}(\Phi)$ admits a radially symmetric ground state. 

Thanks to the radial symmetry of $\Om_0^*$, we can separate the radial and angular variables. To see this, if a function $u$ is given in polar coordinates $(r,\theta)$ as $u(r,\theta)=f(r)\mathrm{e}^{\i m\theta}$ with $m$ an integer, then
\[
    \|(-\i\nabla-\FAB_\Phi)u\|^2_{L^2(\Om_0^*)}=2\pi q_m(f),\quad
    \|u\|_{L^2(\Om_0^*)}^2=2\pi\int_{R_0}^{R_1}|f(r)|^2r\d r,
\]
where 
\[q_m(f)=\int_{R_0}^{R_1}\Bigl(|f'(r)|^2+V_m(r)|f(r)|^2\Bigr)r\d r,\quad V_m(r)=(\Phi-m)^2/r^2.\]
More generally, we can decompose any given  $u\in L^2(\Om_0^*)$ as a Fourier series in the angular variable,
\[u=\sum_{m\in\Z}u_m(r)\mathrm{e}^{\i m\theta},\]
which in turn yields
\[
    \|(-\i\nabla-\FAB_\Phi)u\|^2_{L^2(\Om_0^*)}=2\pi \sum_{m\in\Z}q_m(u_m),\quad
    \|u\|_{L^2(\Om_0^*)}^2=2\pi\sum_{m\in\Z}\int_{R_0}^{R_1}|u_m(r)|^2r\d r.
\] Consequently, using \eqref{LaOmo},  we get
\[\la^{\Om_0^*}(\Phi)=\inf_{m\in\Z}\mu_m,\]
where $\mu_m$ is the lowest eigenvalue of the operator $\ell_m$ in $L^2((R_0,R_1);r\d r)$ defined as
\[\ell_m=-\frac{\d^2}{\d r^2}-\frac1r\frac{\d}{\d r}+V_m(r)\]
with Dirichlet boundary condition on $r=R_0$ and Neumann boundary condition on $r=R_1$.

\begin{lemma}\label{lem:radial-u}
    There exists a ground state  $u$ of $\la^{\Om_0^*}(\Phi),$ as defined in \eqref{LaOmo}, such that  $u$ is radially symmetric, positive valued and  strictly increasing with respect to the radial variable.
\end{lemma}
\begin{proof}
We can express $\mu_m$ in the variational form
\[\mu_m=\inf\{q_m(f)\colon f(R_0)=0,\|f\|_{L^2((R_0,R_1);r\d r)}=1\}.\]
Knowing that $\Phi\in[0,1/2]$, we have $V_m(r)\geq V_0(r)$ and consequently $\inf_{m\in\Z}\mu_m=\mu_0$. This proves that $\la^{\Om_0}(\Phi)=\mu_0$ with a corresponding ground state $f(|x|)$, where $f$ is a ground state of $\ell_0$. Since $q_0(f)\geq q_0(|f|)$, we get that $u=|f|$ is also a ground state of $\ell_0$, and hence a radial non-negative ground state of $\la^{\Om_0^*}(\Phi)$. By Sturm-Liouville theory, $u$ is positive on $(R_0,R_1)$.

    It remains to prove that $u$ is strictly increasing, that is $u'> 0$ on $(R_0,R_1)$. Notice that $u$ satisfies
    \begin{equation}\label{eq:radial-u}
    \left\{\begin{aligned}
        -u''-\frac{1}{r}u'+\frac{\Phi^2}{r^2}u &=\mu_0 u\;\;\text{in}\; (R_0,R_1),\\
        u'(R_1)=0,\; & u(R_0)=0.
    \end{aligned}\right.
    \end{equation}
As in \cite[Proof of Theorem~28]{ColbProvenSavo}, we prove that  $N(r):=ru'(r)$ is positive. Notice that $N(R_1)=0$ and that
\[N'(r)=\left(\frac{\Phi^2}{r^2}-\mu_0\right)ru(r)\quad\mbox{on }(R_0,R_1).\]
Since $u$ is positive on $(R_0,R_1)$ and $u(R_0)=0$, we have $u'(R_0)>0$ and hence $N(R_0)>0$. 

We introduce $r_*=\Phi/\sqrt{\mu_0}$ and consider two cases:
\begin{itemize}
    \item Case I: $r_*\in(R_0,R_1)$:
    \item Case II: $r_*\not\in(R_0,R_1)$.
\end{itemize}
Consequently, 
$r\in(R_0,R_1)$ is a critical point of $N$ if and only if $r=r_*$ and $r_*=\Phi/\sqrt{\mu_0}\in(R_0,R_1)$.

We claim that, in Case II, we have $N(r)>0$ on $(R_0,R_1)$. In fact, in Case II, $N$ does not have critical points in $(R_0,R_1)$ and consequently $N'$ is either positive on negative on $(R_0,R_1)$. If $N'$ were positive, then $N(R_0)\leq N(R_1)=0$ which violates $N(R_0)>0$. Therefore, in Case II, $N'<0$  and we have $N(r)>N(R_1)=0$. 

Now we claim that in Case I, we also have $N(r)>0$ on $(R_0,R_1)$. In fact, studying the sign of $N'$, we get that $N(r)$  increases from $N(R_0)>0$ to $N(r_*)$, and then it decreases from $N(r_*)$ to $N(R_1)=0$. 
\end{proof}

Next, we state a technical lemma that is needed to prove our main result (see \cite[Theorem~29]{ColbProvenSavo}).
\begin{lemma}\label{lem:mono-grad-u}    Assume that $\Phi\in [0,1/2]$. Suppose $u$ is a (real) radial ground state solution for $\la^{\Om_0^*}(\Phi).$ Define 
    \[U(r)=|u'(r)|^2+\frac{\Phi^2}{r^2}u(r)^2,\;\text{ for }\;r\in (R_0,R_1).\]
    Then $U$ is strictly decreasing on $(R_0,R_1).$
\end{lemma}
\begin{proof}
It is sufficient to showing that $U'< 0$ in $(R_0,R_1)$. Indeed, with the help of \eqref{eq:radial-u} and using the fact that $u> 0,\;u'> 0$ on $(R_0,R_1)$ (see Lemma \ref{lem:radial-u}),  it is easy to observe that
\begin{align}
    U'(r) &=2u'(r)u''(r)-\frac{2\Phi^2}{r^3}u(r)^2+\frac{2\Phi^2}{r^2}u(r)u'(r)\\
    &= -\frac{2}{r}(u')^2-\underbrace{2\la^{\Om_0^*}(\Phi) u(r)u'(r)}_{\geq 0}-\frac{2\Phi^2}{r^3}u(r)^2+\frac{4\Phi^2}{r^2}u(r)u'(r)\\
    &\leq-\frac{2}{r}(u')^2 -\frac{2\Phi^2}{r^3}u(r)^2+\frac{4\Phi^2}{r^2}u(r)u'(r)\\
    &= -2\frac{u'(r)^2}{r}(1-\Phi^2)-\frac{2\Phi^2}{r^3}\left(u(r)-ru'(r)\right)^2  \\
    &\leq -2\frac{u'(r)^2}{r}(1-\Phi^2).
\end{align}
Since $\Phi\in [0,\frac{1}{2}],$ we have $1-\Phi^2> 0$ and hence from the above estimate,  it follows that $U'(r)< 0.$
\end{proof}

\section{Proof of Theorem \ref{theo:FB}}\label{sec:proof_thm}

Suppose that $\Omega_0$ is not an annulus and that \eqref{eq:constraint} and \eqref{eq:cond-geom} hold. Without loss of generality, we assume that $x_0=\mathbf 0$ and consequently $\om=\om^*$. Due to the even symmetry and periodicity of $\la^{\Om_0}(\Phi)$ with respect to $\Phi$, we may restrict to $\Phi\in [0,1/2]$, which we assume henceforth.

Let $u$ be the (real) radial ground state of $\la^{\Om_0^*}(\Phi)$ as  in Lemma \ref{lem:radial-u}. We introduce the radial variable $r=|x|$ and with slight abuse of notation we write $u(x)=u(r)$. Define the function $f:\R^2\rightarrow\R$ by 
\begin{equation}\label{eq:def-trial-state}
    f(r)=\begin{cases}
        0 &\text{if}\;0\leq r\leq R_0,\\
        u(r) &\text{if}\;R_0<r<R_1,\\
        u(R_1) &\text{if}\;r\geq R_1.
    \end{cases}
\end{equation}
We can use $f$ as a trial state in the variational characterization of $\la^{\Om_0}(\Phi)$ in \eqref{LaOmo} since $\om=\om^*$ and $f$  satisfies the Dirichlet condition on $\partial\om$. 
\begin{lemma}\label{lem:proof-thm-1}
    Let $f$ be as in \eqref{eq:def-trial-state}. Then,
    \[\la^{\Om_0}(\Phi)\leq \frac{\|(-\i\nabla -\FAB_{\Phi})f\|^2_{L^2(\Om_0)}}{\|f\|_{L^2(\Om_0)}^2}.\]
\end{lemma}

Next we estimate the $L^2$ norm of $f$.
\begin{lemma}\label{lem:proof-thm-2}
    Let $f$ be as in \eqref{eq:def-trial-state}. Then,
    \[\|f\|_{L^2(\Om_0)}^2\geq\|u\|_{L^2(\Om_0^*)}^2.\]
    Moreover, the above inequality is strict if $\Om_0$ is not an annulus.
\end{lemma}
\begin{proof}
    Let $(\Om^*_0)^c$ denote the complement of $\Om_0^*$. We decompose the integral over $\Om$ as
    \[\|f\|_{L^2(\Om_0)}^2=\int_{\Om^*_0\cap \Om_0}f^2\d x+\int_{(\Om^*_0)^c\cap\Om_0}f^2\d x.\]
     Since $f$ vanishes on $\om^*$ and $f=u(R_1)$ on $(\Om^*)^c$, we get
    \[\|f\|_{L^2(\Om_0)}^2=\int_{\Om^*_0\cap \Om}u^2\d x+u(R_1)^2|(\Om^*_0)^c\cap\Om|.\]
    Similarly, we decompose the integral over $\Om_0^*$ and write,
    \[\|u\|_{L^2(\Om^*_0)}^2=\int_{\Om_0\cap\Om_0^*}u^2\d x+\int_{\Om^c_0\cap\Om^*_0}u^2\d x.\]
Knowing that $u\leq u(R_1)$ by Lemma~\ref{lem:radial-u} and that $\Om^*_0\subset\Om^*$, we get
\[\int_{\Om^c_0\cap\Om^*_0}u^2\d x\leq u(R_1)^2|\Om^c_0\cap \Om^*|.\] To finish the proof, we notice that by 
 \eqref{eq:constraint} we have  $|\Om_0\cap(\Om^*_0)^c|=|\Om^*_0\cap\Om^c_0| $. Furthermore, if $\Om_0$ is not the annulus (i.e., equivalently $\Om\neq \Om^*$), then we have $u<u(R_1)$ in $\Om^c_0\cap \Om_0^*$ in view of Lemma \ref{lem:radial-u}. Consequently, we get the strict inequality in this case. 
\end{proof}
Now we proceed in estimating the $L^2$ norm of the magnetic gradient of $f$.
\begin{lemma}\label{lem:proof-thm-3}
    Let $f$ be as in \eqref{eq:def-trial-state}. Then,
    \[\|(-\i\nabla -\FAB_{\Phi})f\|^2_{L^2(\Om_0)}\leq \int_{\Om_0^*} \left(|\nabla u|^2+\frac{\Phi^2}{r^2}u^2\right)\dx.\] 
\end{lemma}
\begin{proof}
Since $f$ is radial, we have
\[
\begin{aligned}
    \|(-\i\nabla -\FAB_{\Phi})f\|^2_{L^2(\Om_0)}&=\int_{\Om_0} \left(|\nabla f|^2+\frac{\Phi^2}{r^2}f^2\right)\dx\\
    &=\int_{\Om^*_0\cap\Om_0} \left(|\nabla f|^2+\frac{\Phi^2}{r^2}f^2\right)\dx+\int_{(\Om^*_0)^c\cap \Om_0} \left(|\nabla f|^2+\frac{\Phi^2}{r^2}f^2\right)\dx\\
    &=\int_{\Om^*_0\cap\Om_0} \left(|\nabla u|^2+\frac{\Phi^2}{r^2}u^2\right)\dx+\int_{(\Om^*_0)^c\cap \Om_0} \frac{\Phi^2}{r^2}u(R_1)^2\dx.
\end{aligned}\]
In a similar fashion, we write
\[
\int_{\Om_0^*} \left(|\nabla u|^2+\frac{\Phi^2}{r^2}u^2\right)\dx=
\int_{\Om_0\cap \Om_0^*} \left(|\nabla u|^2+\frac{\Phi^2}{r^2}u^2\right)\dx+\int_{\Om^c_0\cap\Om_0^*}\left(|\nabla u|^2+\frac{\Phi^2}{r^2}u^2\right)\dx.
\]
Notice that
\[|\nabla u|^2+\frac{\Phi^2}{r^2}u^2=U(r),\]
where $U$ is the  decreasing function introduced in Lemma~\ref{lem:mono-grad-u}. Thus
\[|\nabla u|^2+\frac{\Phi^2}{r^2}u^2\geq U(R_1)=\frac{\Phi^2}{R_1^2}u(R_1)^2\mbox{ for }R_0<r<R_1. \]
Consequently,
\[\int_{\Om^c_0\cap\Om_0^*}\left(|\nabla u|^2+\frac{\Phi^2}{r^2}u^2\right)\dx\geq |\Om^c_0\cap\Om_0^*|\frac{\Phi^2}{R_1^2}u(R_1)^2,\]
which finishes the proof.
\end{proof}

\begin{proof}[Proof of Theorem \ref{theo:FB}]
Collecting Lemmas~\ref{lem:proof-thm-1}, \ref{lem:proof-thm-2} and \ref{lem:proof-thm-3}, we get
\[\lambda^{\Om_0}(\Phi)\leq \frac{\|(-\i\nabla-\FAB_\Phi)u\|_{L^2(\Om_0^*)}^2}{\|u\|_{L^2(\Om_0^*)}^2}=\lambda^{\Om_0^*}(\Phi).\]
It remains to notice that if $\Om_0$ is not an annulus, then one has the strict inequality in Lemma \ref{lem:proof-thm-2} and we get $\lambda^{\Om_0}(\Phi)<\lambda^{\Om_0^*}(\Phi).$ 
\end{proof}

\section{The Neumann problem}\label{sec:Neumann}
This section is devoted to the proof of similar results for the (magnetic) Neumann problem on doubly connected domains, with disk shaped holes.
Specifically, we introduce the lowest Neumann eigenvalue
\begin{equation}\label{LaOmo_N}
    \la^{N,\Om_0}(\Phi)=\inf_{\substack{u\in H^1(\Om_0)}} \frac{\|(-\i\nabla -\Fb_\Phi)u\|^2_{L^2(\Om_0)}}{\|u\|_{L^2(\Om_0)}^2}.
\end{equation}

We begin with showing the existence of a (real) radial ground state solution and establish the monotonicity of it, which is analogous to Lemma \ref{lem:radial-u}. As earlier, we assume $\Phi\in (0,1/2]$.

\begin{lemma}\label{lem:radial-u-N}
    Let $\Om_0^*$ be as given in \eqref{eq:def-opt-dom} and $\la^{N,\Om_0^*}(\Phi)$ be as defined in \eqref{LaOmo_N}. Then there exists a (real) ground state  $u$ of $\la^{N,\Om_0^*}(\Phi)$ such that  $u$ is positive, radially symmetric, and  strictly increasing with respect to the radial variable.
\end{lemma}
\begin{proof}
    Following the same arguments as in the proof of Lemma \ref{lem:radial-u}, we get that there exists a (real) ground state $u$, which is positive on $(R_0,R_1)$ and radial. 

    Now, we show that $u'>0$ on $(R_0,R_1).$ Since $u$ is radial, it satisfies the following equation on $(R_0,R_1)$:
    \begin{equation}\label{eq:radial-u-N}
    \left\{\begin{aligned}
        -u''-\frac{1}{r}u'+\frac{\Phi^2}{r^2}u &=\lambda u\;\;\text{in}\; (R_0,R_1),\\
        u'(R_0)=0,\; & u'(R_1)=0,
    \end{aligned}\right.
    \end{equation}
    where $\lambda=\la^{N,\Om_0^*}(\Phi)$.
    Define $N(r)=ru'(r)$ for $r\in (R_0,R_1).$ It is sufficient to showing that $N$ is positive on $(R_0,R_1).$ Notice that $N(R_0)=N(R_1)=0.$ We have 
    \[N'(r)=\left(\frac{\Phi^2}{r^2}-\lambda\right)ru(r)\quad\mbox{for }r\in (R_0,R_1).\]
    Let $r_*=\Phi/\sqrt{\lambda}.$ We claim that $r_*\in (R_0,R_1).$ Otherwise, we have two cases: either $N'(r)>0$ or $N'(r)<0.$ In the former case, we get by using the boundary conditions that $0=N(R_0)<ru'(r)<N(R_1)=0,$ in $(R_0,R_1),$ which is impossible. Similarly, we exclude the latter case also. Thus we have $r_*\in (R_0,R_1),$ and this is the only point of maximum of $N$ on $(R_0,R_1).$ It is easy to observe that $N$ is strictly increasing on $(R_0,r_*)$ and strictly decreasing on $(r_*,R_1).$ Therefore, we have $ru'(r)>0$ for $r\in (R_0,R_1).$ Consequently, we get $u'>0$ in $(R_0,R_1).$
\end{proof}

\begin{remark}\label{rem:mono-grad-u-N}    Let $u$ be a (real) radial ground state of $\la^{N,\Om_0^*}(\Phi)$. Assume that $\Phi\in (0,1/2]$.  Define 
    \[U(r)=|u'(r)|^2+\frac{\Phi^2}{r^2}u(r)^2,\;\text{ for }\;r\in (R_0,R_1).\]
    Then using the monotonicity of $u$ derived in Lemma~\ref{lem:radial-u-N}, one can easily prove (see Lemma~\ref{lem:mono-grad-u}) that $U$ is strictly decreasing on $(R_0,R_1).$
\end{remark}

 Recall that $\Om_0=\Om\setminus \overline{\om}$ and $\Om_0^*=\Om^*\setminus \overline{\om^*}$ satisfying \eqref{eq:constraint} and \eqref{eq:cond-geom}. Without loss of generality, we assume that $\mathbf 0\in \om$ and hence, we take $\om=\om^*$. 

Let $u$ be the (real) radial ground state of $\la^{N,\Om_0^*}(\Phi)$ given by Lemma \ref{lem:radial-u-N}. We write $u(x)=u(r)$ for brevity. Define the function $f:\R^2\setminus \overline{\om^*}\rightarrow\R$ by 
\begin{equation}\label{eq:def-trial-state-N}
    f(r)=\begin{cases}
        u(r) &\text{if}\;R_0<r<R_1,\\
        u(R_1) &\text{if}\;r\geq R_1.
    \end{cases}
\end{equation}
Observe that $f|_{\Om_0}$ is a valid trial state in the variational characterization of $\la^{N,\Om_0^*}(\Phi)$ in \eqref{LaOmo_N}.  With this trial state $f$, we now prove the following estimate about the $L^2$ norm.

\begin{lemma}\label{lem:proof-thm-N}
    Let $f$ be as in \eqref{eq:def-trial-state-N}. Then, we have
    \[\|f\|_{L^2(\Om_0)}^2\geq\|u\|_{L^2(\Om_0^*)}^2.\]
    Moreover, the above inequality is strict if $\Om_0$ is not an annulus.
\end{lemma}
\begin{proof}
     We decompose the integral over $\Om_0$ as follows:
    \[
\|f\|_{L^2(\Om_0)}^2 =\int_{\Om_0^*\cap \Om_0}f^2\d x+\int_{(\Om_0^*)^c\cap\Om_0}f^2\dx.
    \]
Since $\om=\om^*$, we have    $f=u(R_1)$ on $(\Om^*_0)^c\cap\Om_0$. Thus we get
    \[\|f\|_{L^2(\Om_0)}^2=\int_{\Om^*_0\cap \Om_0}u^2\d x+u(R_1)^2|(\Om^*_0)^c\cap\Om_0|.\]
    Similarly, decomposing the integral over $\Om_0^*$, one has
    \[\|u\|_{L^2(\Om^*_0)}^2=\int_{\Om_0^*\cap \Om_0}u^2\d x+\int_{\Om^*_0\cap \Om_0^c}u^2\d x.\]
To finish the proof, we use that  $ u\leq u(R_1)$ on $(R_0,R_1)$ in view of Lemma \ref{lem:radial-u-N}, and that $|(\Om_0^*)^c\cap\Om_0|=|\Om^*_0\cap \Om_0^c|$ in view of \eqref{eq:constraint}.
\end{proof}

We estimate the magnetic gradient norm of $f$ in a similar fashion to Lemma~\ref{lem:proof-thm-3}.
\begin{lemma}\label{lem:Neu-2}
    Let $f$ be as defined in \eqref{eq:def-trial-state-N}. Then the following holds:
    \[\|(-\i\nabla -\FAB_{\Phi})f\|^2_{L^2(\Om_0)}\leq \|(-\i\nabla -\FAB_{\Phi})u\|^2_{L^2(\Om_0^*)}.\]
\end{lemma}

Using Lemma \ref{lem:proof-thm-N} and Lemma \ref{lem:Neu-2} in \eqref{LaOmo_N}, we have the following theorem.

\begin{theorem}\label{theo:Neu}
    Let $\Om_0$ and $\Om_0^*$ be as given in \eqref{eq:def-opt-dom}, and suppose that \eqref{eq:constraint} and \eqref{eq:cond-geom} hold. If $\Om_0$ is not an annulus, then for a given $\Phi\in (0,1/2]$, we have 
    \begin{equation}
        \la^{N,\Om_0}(\Phi)<\la^{N,\Om_0^*}(\Phi).
    \end{equation}
\end{theorem}

\begin{remark}
    Observe that  Theorem \ref{theo:Neu} extends to all $\Phi\not\in \mathbb{Z}$ due to periodicity and even symmetry of $\la^{N,\Om_0}(\Phi)$ with respect to $\Phi$ (see Section \ref{sec:Main}). For $\Phi\in\Z$, $\la^{N,\Om_0}(\Phi)$ vanishes because we reduce to the case of the Neumann Laplacian without magnetic field.
\end{remark}

\subsection*{Acknowledgments} The authors would like to thank Vladimir Lotoreichik for the valuable comments. AK is partially supported by CUHK-SZ grant no. UDF01003322. MG is grateful to TIFR Centre for Applicable Mathematics (TIFR-CAM) for the financial support and for providing various academic resources.

\bibliographystyle{abbrvurl}
\bibliography{Reference}
\end{document}